\documentclass[11pt]{amsart}

\usepackage{latexsym,amssymb,amsthm,amsmath,amscd,multirow,graphics,multicol}
\usepackage{amssymb} 

\theoremstyle{plain}  
\newtheorem{theorem}{Theorem}[section]

\newtheorem{corollary}{Corollary}[section]

\newtheorem*{question}{Question}
 
\newtheorem{example}{Example}[section]
\numberwithin{equation}{section}

\theoremstyle{remark}
\newtheorem{remark}{Remark}[section]

 \numberwithin{equation}{section}
\def\<{\left < }
\def\>{\right >}
\def\({\left ( }
\def\){\right )}
\def\e{\eqref}
\def\e{\eqref}
\def\p{\partial }

\def\o{\omega }

\def\M{M^2_1 }

\def\x{\text{\small$\frac{\partial}{\partial x}$}}
\def\y{\text{\small$\frac{\partial}{\partial y}$}}

\def\sech{\hskip.02in {\rm sech}}

\begin{document}

\title[Minimal Lorentz surfaces in  indefinite space forms]
{Classification of minimal Lorentz surfaces in  indefinite space forms  with arbitrary  codimension and arbitrary index}

\author[B.-Y. Chen]{Bang-Yen Chen}

 \address{Department of Mathematics, 
	Michigan State University, East Lansing, MI 48824, U.S.A.}
	
\email{bychen@math.msu.edu}

\begin{abstract} Since J. L. Lagrange  initiated in \cite{L} the study of minimal surfaces of Euclidean 3-space in 1760,  minimal surfaces in real space forms have  been studied extensively  by many mathematicians during the last two and half centuries. In contrast, so far very few results on minimal Lorentz surfaces in indefinite space forms are known. Hence, in this paper we investigate minimal Lorentz surfaces in arbitrary indefinite space forms. As a consequence, we obtain several classification results for minimal Lorentz surfaces in indefinite space forms. In particular, we completely classify all minimal Lorentz surfaces in a pseudo-Euclidean space $\mathbb E^m_s$ with arbitrary dimension $m$ and arbitrary index $s$.
\end{abstract}


 
\maketitle

\section{Introduction.}

Let $\mathbb E_{s}^m$ denote the  pseudo-Euclidean $m$-space with the canonical  metric of  index $s$ given by
\begin{align}\label{1.1}g_{0}= -\text{\small$\sum$}_{i=1}^{s} dx_{i}^{2} + \text{\small$\sum$}_{j=s+1}^{m}dx_{j}^{2}, \end{align}
where $(x_{1},\ldots,x_{m})$ is a rectangular coordinate system of  $\mathbb E_{s}^m$. 
The {\it light cone} ${\mathcal LC}$ of $\mathbb E^{m+1}_s$ is defined by ${\mathcal LC}=\{{\bf x}\in \mathbb E^{m+1}_s:\<{\bf x},{\bf x}\>=0\}$.
  
  We put
\begin{align} &\label{1.2} S^k_s(c)=\{x\in \mathbb E^{k+1}_s| \<x,x\>={c}^{-1}>0\},\\&\label{1.3} H^k_s(-c)=\{x\in \mathbb E^{k+1}_{s+1}| \<x,x\>=-c^{-1}<0\},\end{align}
where $\<\;,\:\>$ denotes the indefinite inner product on $\mathbb E^{k+1}_t$. The
$S^k_s(c)$ and $H^k_s(-c)$ are complete pseudo-Riemannian manifolds with index $s$ and of constant curvature $c$ and $-c$, respectively. 

The $S^k_s(c)$ and $H^k_s(-c)$ are called {\it pseudo $k$-sphere} and {\it pseudo hyperbolic $k$-space}, respectively. The pseudo-Riemannian manifolds $\mathbb E^k_s, S^k_s$ and $H^k_s$  are known as the {\it indefinite space forms}. In particular, $\mathbb E^k_1, S^k_1$ and $H^k_1$ are called Minkowski, de Sitter and anti-de Sitter spacetimes, which play very important roles in relativity theory. 
    
    The history of minimal surfaces goes back to J. L. Lagrange (1736-1813) who initiated in 1760  the  study of minimal surfaces in Euclidean 3-space (see \cite{L}). Since then  minimal surfaces have attracted many mathematician. In particular, minimal surfaces in real space forms have  been studied very extensively during the last two and half centuries (see, \cite[pages 207--249]{c3} and \cite{Ni,Os} for details).
    
    In \cite{VP1,VP2}, L. Verstraelen and M. Pieters studied some families of Lorentz surfaces in 4-dimensional indefinite space forms with index 2. Recently, parallel Lorentz surfaces in indefinite space forms with arbitrary codimension and arbitrary index  were completely classified in a series of articles  \cite{c6}-\cite{cv} (see also \cite{B,g1,g2,Ma}). Moreover, Lorentz surfaces with parallel mean curvature vector in an arbitrary pseudo-Euclidean space were classified in \cite{c11} (see also \cite{HY}).
    Further,  minimal Lorentz surfaces in Lorentzian complex space forms $\tilde M^2_1(c)$ with complex index one were investigated in \cite{c4,c5}. 

  In this paper, we study minimal Lorentz surfaces in indefinite space forms $R^m_s(c)$ with arbitrary codimension and arbitrary index s. In particular, we completely classify minimal Lorentz surfaces in an arbitrary pseudo-Euclidean space in section 4. In section 5, we classify  minimal Lorentz surfaces of constant curvature one in an arbitrary  pseudo $m$-sphere $S^m_s(1)$. The classification of minimal Lorentz surfaces of constant curvature $-1$ in a pseudo-hyperbolic $m$-space $H^m_s(-1)$ are obtained in section 6. In the last two sections, we provide many explicit examples of  minimal Lorentz surfaces in $S^m_s(1)$ and in $H^m_s(-1)$. 
 
\section{Basics formulas, equations and definitions.}

 Let  $R^m_s(c)$ be an $m$-dimensional indefinite space form of constant sectional curvature $c$ and with index $s$. 
The curvature tensor  of $R^m_s(c)$ is given by
\begin{align} &\label{2.1}\tilde R(X,Y)Z=c\{\left<Y,Z\right>X -  \left<X,Z\right>  Y \}.\end{align}
   
    Let $\psi:M^2_1 \to R^m_s(c)$ be an isometric immersion of a  Lorentz surface $M^2_1$ into $R^m_s(c)$.       
Denote by $\nabla$ and $\tilde\nabla$ the Levi-Civita connections on $M^2_1$ and $\tilde R^m_s(c)$, respectively. 
Let $X,Y$ be vector fields tangent to $M^2_1$ and  $\xi$  normal to $M^2_1$ in $R^m_s(c)$. The formulas of Gauss and Weingarten  are given by (cf. \cite{c1,c2,O}):
\begin{align} &\label{2.2}\tilde \nabla_XY=\nabla_XY+h(X,Y), \;\;
\\& \label{2.3}\tilde \nabla_X \xi=-A_\xi X+D_X\xi.\end{align} These formulas define $h$, $A$ and $D$, which are called the second
fundamental form, the shape operator and the normal connection, respectively. 

For each normal vector $\xi\in T_x^{\perp}M^2_1$, the shape operator $A_{\xi}$ at $\xi$ is a symmetric endomorphism of the tangent space $T_xM^2_1$,  $x\in M^2_1$. The shape operator and the second fundamental form are related by
\begin{align}\label{2.4} \<h(X,Y),\xi\>=\<A_{\xi}X,Y\>.\end{align}
 The mean curvature vector $H$ of $M^2_1$ in $R^m_s(c)$ is defined by \begin{align}\label{2.5} H=\text{\small$ \frac{1}{2}$} {\rm trace}\, h.\end{align}
 A Lorentz surface in an indefinite space form is called {\it totally geodesic} if its second fundamental form  vanishes identically.  It  is called  {\it minimal } if its mean curvature vector vanishes identically.

The equations of Gauss, Codazzi and Ricci are given
respectively by
\begin{align} &\label{2.6} R(X,Y)Z =c\{\left<Y,Z\right>X -  \left<X,Z\right>  Y \} + A_{h(Y,Z)}X- A_{h(X,Z)}Y,\\
&  \label{2.7} (\overline \nabla_X h)(Y,Z)=(\overline\nabla_Y h)(X,Z),\\ \label{2.8}
& \<R^{D}(X,Y)\xi,\eta\> = \<[A_{\xi},A_{\eta}]X,Y\>,\end{align}
for vector fields $X,Y,Z$  tangent to $M^2_1$, $\xi$  normal to $M^2_1$, where $\overline\nabla h$ is defined by
\begin{equation}\label{2.9}(\overline\nabla_X h)(Y,Z) = D_X h(Y,Z) - h(\nabla_X Y,Z) - h(Y,\nabla_X Z),\end{equation} and $R^D$ is the curvature tensor associated to the normal connection $D$, i.e., 
\begin{align}\label{2.10} &R^D(X,Y)\xi=D_XD_Y\xi-D_Y D_X\xi-D_{[X,Y]}\xi.\end{align}

A  vector $v$ in $R^m_s(c)$ is called {\it spacelike} (respectively, {\it timelike}, or {\it light-like}) if  $\<v,v\>>0$ (respectively, $\<v,v\><0$, or $\<v,v\>=0$ and $v\ne 0$).
A curve $z(x)$ in $R^m_s(c)$ is called spacelike (respectively, timelike or null) if its velocity vector $z'(x)$ is spacelike (respectively,  timelike or lightlike) at each point.

 \section{A special coordinate system on a Lorentz surface.} 
 
 Let $\M$ be a  Lorentz surface.  We may choose a local coordinate system $\{x,y\}$ on  $M^2_1$ such that the metric tensor is given by
\begin{align}\label{3.1} g=-E^2(x,y)(dx\otimes dy+dy\otimes dx)\end{align}
for some positive function $E$.

The Levi-Civita connection of $g$ satisfies
\begin{align}\label{3.2}\nabla_{\frac{\p}{\p x}}\frac{\p}{\p x}=\frac{2 E_x}{E}  \frac{\p}{\p x}, \; \nabla_{\frac{\p}{\p x}}\frac{\p}{\p y}=0, \;  \nabla_{\frac{\p}{\p y}}\frac{\p}{\p y}=\frac{2 E_y} {E} \frac{\p}{\p y}\end{align}
and the Gaussian curvature $K$  is given by
\begin{align}\label{3.3} K=\frac{2E E_{xy}-2E_x E_y}{E^4}.\end{align}

If we put \begin{align}\label{3.4} e_1=\frac{1}{E}\frac{\partial}{\partial x} ,\;\; e_2=\frac{1}{E}\frac{\partial}{\partial y},\end{align} then $\{e_1,e_2\}$ forms a pseudo-orthonormal frame satisfying 
\begin{align} \label{3.5} & \<e_1, e_1\>=\< e_2,e_2\>=0,\;\<e_1, e_2\>=-1.\end{align} 

We define the connection 1-form  $\omega$ by the following equations:
\begin{align}\label{3.6} &\nabla_Xe_1=\omega(X)e_1,\, \;\; \nabla_Xe_2=-\omega(X)e_2.\end{align}
 From \e{3.2} and \e{3.4} we find
\begin{equation}\begin{aligned}\label{3.7}&\nabla_{e_1}e_1=\frac{E_x}{E^2}e_1, \; \nabla_{e_2}e_1=-\frac{E_y}{E^2}e_1, \; \\& \nabla_{e_1}e_2=-\frac{E_x}{E^2}e_2, \;  \nabla_{e_2}e_2=\frac{E_y}{E^2}e_2 .\end{aligned}\end{equation}
By comparing \e{3.6} and \e{3.7}, we get
\begin{align}\label{3.8}  \o(e_1)=\frac{E_x}{E^2},\;\; \o(e_2)=-\frac{E_y}{E^2}.  \end{align} 

Let $\psi:\M\to R^m_s(c)$ be an isometric immersion of $\M$ into $R^m_s(c)$. Then
it follows from \e{2.5} and \e{3.5} that the mean curvature vector of $\M$ is given by
\begin{align}\label{3.9}  H=-h(e_1,e_2).\end{align} 
Therefore,  $\M$ is a minimal surface of $R^m_s(c)$ if and only if $h(e_1,e_2)=0$  holds identically.

\section{Minimal Lorentz surfaces in  $\mathbb E^m_s$.}

In this section, we completely classify minimal Lorentz surface  in an arbitrary pseudo-Euclidean $m$-space $\mathbb E^m_s$. More precisely, we prove the following.

\begin{theorem} \label{T:4.1} A Lorentz surface in  a pseudo-Euclidean m-space $\mathbb E^m_s$  is  minimal if and only if locally the immersion takes the form $L(x,y)=z(x)+w(y)$, where $z$ and  $w$ are null curves satisfying $\<z'(x),w'(y)\>\ne 0$.
\end{theorem}
\begin{proof} Let  $L:M^2_1\to \mathbb E^m_s$ be an isometric  immersion of a Lorentz surface $\M$ into a pseudo-Euclidean m-space $\mathbb E^m_s$ with index $s\geq 1$. 
We choose a local coordinate system $\{x,y\}$ on  $M^2_1$ satisfying 
\begin{align}\label{4.1} g=-E^2(x,y)(dx\otimes dy+dy\otimes dx)\end{align}
Then we have \e{3.2}-\e{3.9}.

If $\M$ is a minimal surface, it follows from \e{3.9} that  $h(e_1,e_2)=0$ holds. Hence, we may put \begin{align}\label{4.2} h(e_1,e_1)=\xi,\: \; h(e_1,e_2)=0,\; \:h(e_2,e_2)=\eta\end{align}
for some normal vector  fields $\xi,\eta$.  After applying \e{2.2},  \e{3.2}, \e{3.4}, and \e{4.2}, we obtain
\begin{equation}\begin{aligned}\label{4.3} & L_{xx}=\frac{2E_x}{E}L_x+  E^2 \xi,\;\; L_{xy}=0,\;\; L_{yy}=\frac{2E_y}{E}L_y+E^2 \eta.\end{aligned} \end{equation}
After solving the second equation in  \e{4.3}, we find
\begin{align} \label{4.4}  &L(x,y)=z(x)+w(y)\end{align}
for some  vector-valued functions $z(x),w(y)$.  Thus, by  applying \e{3.1} and \e{4.4}, we obtain $\<z',z'\>=\<w',w'\>=0,$ and $\<z',w'\>=-E^2.$ Therefore,  $z$ and $w$ are null curves satisfying $\<z',w'\>\ne 0$.

Conversely, if $L:\M\to \mathbb E^m_s$ is an immersion of a Lorentz surface $\M$ into $\mathbb E^4_s$ such that $L=z(x)+w(y)$ for some null curves $z,w$ satisfying $\<z',w'\>\ne 0$,
then we obtain $\<L_x,L_x\>=\<L_y,L_y\>=0, \<L_x,L_y\>\ne 0,$ and $L_{xy}=0.$ Thus, $\M$ is  surface with induced metric given by
$g=F(x,y)(dx\otimes dy+dy\otimes dx)$ for some nonzero function $F$.
 Moreover, it follows from \e{3.9} and $L_{xy}=0$ that $L$ is a minimal immersion.
\end{proof}

\begin{remark} Flat minimal Lorentz surfaces in the Lorentzian complex plane ${\bf C}^2_1$ have been completely classified in \cite{c4}. Moreover, if $m=3$, this theorem is due to \cite[Theorem 3.5]{Mc}.
\end{remark}

In particular, if $M^2_1$ is a flat Lorentz surface,  we have the following.

\begin{corollary} \label{C:4.1} A flat Lorentz surface in  a pseudo-Euclidean m-space $\mathbb E^m_s$  is  minimal if and only if locally the immersion takes the form \begin{align}\label{4.5}L(x,y)=z(x)+w(y),\end{align} where $z$ and  $w$ are null curves satisfying $\<z',w'\>=constant\ne 0$.
\end{corollary}
\begin{proof} Let  $L:M^2_1\to \mathbb E^m_s$ be an isometric  immersion of a flat Lorentz surface $M^2_1$ into a pseudo-Euclidean m-space $\mathbb E^m_s$.  Then  we may choose a local coordinate system $\{x,y\}$ on  $M^2_1$ satisfying 
\begin{align}\label{4.6} g=-(dx\otimes dy+dy\otimes dx)\end{align}
Then we find from \e{4.3} that the immersion $L$ satisfies
\begin{align}\label{4.7} & L_{xx}= \xi,\;\; L_{xy}=0,\;\; L_{yy}= \eta\end{align}
for some normal vector fields $\xi,\eta$. After solving the second equation in \e{4.7} we find
\begin{align} \label{4.8}  &L(x,y)=z(x)+w(y)\end{align}
for some vector functions $z,w$. Thus, by applying \e{4.6}, we find $\<z',z'\>=\<w',w'\>=0$ and $\<z',w'\>=-1$. Consequently, $z$ and  $w$ are null curves satisfying $\<z',w'\>=constant \ne 0$.

Conversely,  consider a map $L$ defined by $L(x,y)=z(x)+w(y)$, where  $z$ and  $w$ are null curves satisfying $\<z',w'\>=constant \ne 0$ .
Then we have $$\<L_x,L_x\>=\<L_y,L_y\>=0,\;\; \<L_x,L_y\>=constant \ne 0.$$ Thus, with respect to the induced metric, \e{4.5} defines an isometric immersion of a flat Lorentz surface $M^2_1$ into $\mathbb E^m_s$. The remaining follows from Theorem \ref{T:4.1}.
\end{proof}

\section{Minimal Lorentz surfaces in  $S^m_s(1)$.}

Let $\psi:\M\to S^m_s(1)$ be an isometric immersion of a Lorentz surface  into $S^m_s(1)$. Denote by $L=\iota\circ \psi :\M\to \mathbb E^{m+1}_{s}$ the composition  of $\psi$ and the inclusion $\iota:S^m_s(1)\subset \mathbb E^{m+1}_{s}$ via \e{1.2}.

Obviously, every totally geodesic Lorentz surface in an indefinite space form $R^m_s(c)$ is of constant curvature $c$. A natural question is the following:

\begin{question} Besides totally geodesic ones how many minimal Lorentz surfaces  of constant curvature $c$ in $R^m_s(c)$ are there \hskip-.02in ?\end{question}
\vskip.05in 

Theorem 5.1 of \cite{c4} provides the answer to this basic question for $c=0$.

In this section, we give an answer to this question for $c>0$. More precisely, we 
classify all minimal Lorentz surfaces of constant curvature one in the pseudo-sphere $S^m_s(1)$ with arbitrary $m$ and arbitrary index $s$.

\begin{theorem} \label{T:5.1} Let $\M$ be a Lorentz surface of constant curvature one. Then an isometric immersion
$\psi:\M\to S^m_s(1)$ is  minimal if and only if  one of the following three cases occurs:

\vskip.04in

\noindent {\rm (a)}  $\M$ is an open portion of a totally geodesic $S^2_1(1)\subset S^m_s(1)$.
\vskip.04in

\noindent {\rm (b)} The immersion $L=\iota\circ \psi :\M\to S^m_s(1)\subset \mathbb E^{m+1}_s$ is locally given by
\begin{align}\label{5.1} & L(x,y)=\frac{z(x)}{x+y}-\frac{z'(x)}{2},\end{align}
where  $z(x)$ is a spacelike curve with constant speed $2$ lying in the light cone $\mathcal LC$ satisfying $\<z'',z''\>=0$ and $z'''\ne 0$.

\vskip.04in

\noindent {\rm (c)}  The immersion $L=\iota\circ \psi :\M\to S^m_s(1)\subset \mathbb E^{m+1}_s$ is locally given by
  \begin{align}\label{5.2} & L(x,y)=\frac{z(x)+w(y)}{x+y}-\frac{z'(x)+w'(y)}{2},\end{align}
where  $z$ and $w$ are curves in $\mathbb E^{m+1}_s$ satisfying \begin{align} &\tag{c.1} 
\Big<\text{\small$ \frac{z(x)+w(y)}{x+y}-\frac{z'(x)+w'(y)}{2},\frac{z(x)+w(y)}{x+y}-\frac{z'(x)+w'(y)}{2}$}\Big>=1,
\\&\tag{c.2}2 \<z+w,z'''\>=(x+y)\<z'+w',z'''\>,
\\&\tag{c.3} 2\<z+w,w'''\> =(x+y)\<z'+w',w'''\>.
\end{align} 
 \end{theorem}
\begin{proof}  Assume that $\psi: \M\to  S^m_s(1)$ is an isometric immersion of a Lorentz surface $\M$ of constant curvature one into  $S^m_s(1)$. 
If $\M$ is totally geodesic in $S^m_s(1)$,  we obtain case (a). Hence, let us assume that $\M$ is non-totally geodesic in $S^m_s(1)$.  

Since $\M$ is of constant curvature one, we may choose local coordinates $\{x,y\}$ such that the metric tensor is given by
\begin{align}\label{5.3}& g=\frac{-2}{(x+y)^2}(dx\otimes dy+dy\otimes dx).\end{align}
Hence the Levi-Civita connection satisfies
\begin{align}\label{5.4}& \nabla_{\x}\x=\frac{-2}{x+y}\x,\;\; \nabla_{\x}\y=0,\;\; \nabla_{\y}\y=\frac{-2}{x+y}\y.\end{align}

Let us put 
\begin{align}\label{5.5}& \x=\frac{\sqrt{2}e_1}{x+y},\;\; \y=\frac{\sqrt{2}e_2}{x+y}.\end{align} Then we get 
\begin{align}\label{5.6}& \<e_1,e_1\>=\<e_2,e_2\>=0,\;\; \<e_1,e_2\>=-1.\end{align} 
Because $\M$ is  minimal in $S^m_s(1)$,  it follows from \e{3.9} and \e{5.3} that  $h(e_1,e_2)=0$. Hence, we may put
\begin{align}\label{5.7} h(e_1,e_1)=\xi,\; h(e_1,e_2)=0,\; h(e_2,e_2)=\eta \end{align}
for some normal vector fields $\xi,\eta$. Without loss of generality, we may assume $\xi\ne 0$.  Since $K=1$, it follows from the equation \e{2.6} of Gauss and \e{5.3} that $\<\xi,\eta\>=0$.

\vskip.04in
\noindent {\it Case} (i):  $\eta=0$. From formula \e{2.2} of Gauss, \e{5.3}-\e{5.5}, and \e{5.7}, we get
\begin{equation}\begin{aligned}\label{5.8}&L_{xx}=\frac{2\xi }{(x+y)^2}-\frac{2L_x}{x+y} ,\;\; L_{xy}=\frac{2 L}{(x+y)^2},\;\;  L_{yy}=-\frac{2 L_y}{x+y} 
.\end{aligned} \end{equation}
After solving the last two equations in \e{5.8} we obtain
\begin{align}\label{5.9}& L(x,y)=\frac{z(x)}{x+y}-\frac{z'(x)}{2}\end{align}
for some $\mathbb E^{m+1}_s$-valued function $z(x)$. Since the metric tensor is given by \eqref{5.3}, one finds
\begin{align}\notag \<L_x,L_x\>=\<L_y,L_y\>=0\;\; {\rm and}\;\; \<L_x,L_y\>=-\frac{2}{(x+y)^2}.\end{align}
Thus, it follows from \e{5.8}, \e{5.9} and $\<L,L\>=1$ that $z(x)$ satisfies $$\<z,z\>=\<z'',z''\>=0 \;\; {\rm and}\;\; \<z',z'\>=4.$$ Moreover,   by substituting \e{5.9} into  the first equation \e{5.8} we find \begin{align} \label{5.10}\xi=-\frac{(x+y)^2 z'''(x)}{4}.\end{align} 
Combining this with $\xi\ne 0$ gives $z'''(x)\ne 0$.
Consequently, we obtain case (b).

Conversely, suppose that $L$ is given by \e{5.1}, where  $z(x)$ is a spacelike curve with constant speed $2$ lying in the light cone $\mathcal LC$ $\subset \mathbb E^{m+1}_{s}$ satisfying $\<z'',z''\>=0$ and $z'''\ne 0$.  Then $L$ satisfies \e{5.9} with $\xi$ given by \e{5.10}. From the assumption, we have
\begin{align} \label{5.11} \<z,z\>=\<z,z'\>=\<z'',z''\>=0,\;\; \<z',z'\>=-\<z,z''\>=4 .\end{align}
By using \e{5.9} and \e{5.11} we know that  $\<L,L\>=1$ and the induced metric tensor is given by \e{5.3}. Moreover, the second equation in \e{5.8} shows that the second fundamental form of $\psi$ satisfies $h(\x,\y)=0$. Consequently, the immersion $\psi$ is a minimal immersion.

\vskip.04in
\noindent {\it Case} (ii):  $\eta\ne 0$. After applying formula \e{2.2} of Gauss, \e{5.3}-\e{5.5}, and \e{5.7}, we obtain
\begin{equation}\begin{aligned}\label{5.12}&L_{xx}=\frac{2\xi }{(x+y)^2}-\frac{2L_x}{x+y} ,\;\; L_{xy}=\frac{2 L}{(x+y)^2},\;\;  L_{yy}=\frac{2\eta}{(x+y)^2}-\frac{2 L_y}{x+y} .\end{aligned} \end{equation}
The compatibility conditions of \e{5.12} are given by 
\begin{align}\label{5.13}&\tilde\nabla_\y \xi=\frac{2\xi}{x+y},\;\; \tilde \nabla_\x \eta=\frac{2\eta}{x+y}.\end{align}
Solving \e{5.13} gives \begin{align}\label{5.14} \xi=(x+y)^2 A(x),\,\;\;  \eta=(x+y)^2 B(y)\end{align} for some $\mathbb E^{m+1}_s$-valued functions $A(x),B(y)$. 
Substituting \e{5.14} into \e{5.12} yields
\begin{equation}\begin{aligned}\label{5.15}&L_{xx}=A(x)-\frac{2L_x}{x+y} ,\;\; L_{xy}=\frac{2 L}{(x+y)^2},\;\;  L_{yy}=B(y)-\frac{2 L_y}{x+y} .\end{aligned} \end{equation}
After solving system \e{5.15}, we obtain 
\begin{align}\label{5.16}&L(x,y)=\frac{z(x)+w(y)}{x+y} -\frac{z'(x)+w'(y)}{2}, \end{align}
where $z(x),w(y)$ are $\mathbb E^{m+1}_s$-valued functions  satisfying
\begin{align}&\label{5.17} z'''(x)=-4 A(x),\;\; w'''(y)=-4 B(y). \end{align}
From $\<L,L\>=1$ and  \e{5.16}, we obtain condition (c.1) in Theorem \ref{T:5.1}.

By combining \e{5.15} and \e{5.17}, we obtain
\begin{equation}\begin{aligned}\label{5.18}&L_{xx}=-\frac{z'''}{4}-\frac{2L_x}{x+y} ,\;\; L_{xy}=\frac{2 L}{(x+y)^2},\;\;  L_{yy}=-\frac{w'''(y)}{4}-\frac{2 L_y}{x+y} .\end{aligned} \end{equation}
Since the metric tensor of $\M$ is given by \e{5.3}, we find
\begin{align}\label{5.19}&\<L_x,L_x\>=\<L_y,L_y\>=0, \;\;\<L_x,L_y\>=-\frac{2}{(x+y)^2}. \end{align}

Because $\<L,L\>=1$, we have $\<L_{xx},L\>=-\<L_x,L_x\>=0$. Thus, we obtain condition (c.2) from \e{5.16}, \e{5.18} and \e{5.19}

Similarly, due to $\<L_{yy},L\>=-\<L_y,L_y\>=0$, we may also derive condition (c.3) from \e{5.16} and \e{5.18}.

Conversely, assume that $L$ is defined by \begin{align}\label{5.20}&L(x,y)=\frac{z(x)+w(y)}{x+y} -\frac{z'(x)+w'(y)}{2}, \end{align}
 where $z(x),w(y)$ are curves satisfying conditions (c1), (c.2) and (c.3). 
Then it follows from \e{5.20} that $L$ satisfies system \e{5.18}. Also, it follows from \e{5.20} and condition (c.1) that  $\<L,L\>=1$. Thus, we have
\begin{align}\label{5.21}&\<L,L_x\>=\<L,L_y\>=0,\end{align}
which implies that
\begin{align}\label{5.22}&\<L_x,L_x\>=-\<L,L_{xx}\>,\; \<L_x,L_y\>=-\<L,L_{xy}\>,\; \<L_y,L_y\>=-\<L,L_{yy}\>.\end{align}
By applying \e{5.22}, (c.1) and the first equation in \e{5.22}, we obtain 
\begin{align}\label{5.23}&\<L_x,L_x\>=-\<L,L_{xx}\>=\frac{2}{(x+y)^2}\<L_x,L_x\>,\end{align} which shows that $\<L_x,L_x\>=0$. 

Similarly, from \e{5.22} and (c.3) we find $\<L_y,L_y\>=0$.
Also, after applying (c.1), \e{5.22} and the second equation in \e{5.18}, we find $\<L_x,L_y\>=-2/(x+y)^2$. Consequently, the induced metric tensor via $L$ is given by \e{5.3}. Finally, it follows from \e{3.9} and the second equation in \e{5.18} that  $\psi:\M\to S^m_s(1)$ is a minimal immersion.
\end{proof}

\section{Minimal Lorentz surfaces in  $H^m_s(-1)$.}

Let $\psi:\M\to H^m_s(-1)$ be an isometric immersion of a Lorentz surface  into $H^m_s(-1)$. Denote by $L=\iota\circ \psi :\M\to \mathbb E^{m+1}_{s+1}$ the composition  of $\psi$ and the inclusion $\iota:H^m_s(-1)\subset \mathbb E^{m+1}_{s+1}$ via \e{1.2}.  

In this section, we provide the following answer to the basic question proposed in section 5 for $c<0$.

\begin{theorem} \label{T:6.1} Let $\M$ be a Lorentz surface of constant Gauss curvature $-1$. Then an isometric immersion
$\psi:\M\to H^m_s(-1)$ is a minimal immersion if and only if  one of the following three cases occurs:
\vskip.04in

\noindent {\rm (i)} $\M$ is an open portion of a totally geodesic $H^2_1(-1)\subset H^m_s(-1)$.
\vskip.04in

\noindent {\rm (ii)} The immersion $L=\iota\circ \psi :\M\to H^m_s(-1)\subset \mathbb E^{m+1}_{s+1}$ is locally given by
\begin{align}\label{6.1} & L(x,y)=z(x)\tanh\!\Big(\text{\small$\frac{x+y}{\sqrt{2}}$}\Big)-\frac{z'(x)}{\sqrt{2}},\end{align}
where  $z(x)$ is a timelike curve with constant speed $\sqrt{2}$ lying in the light cone $\mathcal LC$ $\subset \mathbb E^{m+1}_{s+1}$ satisfying $\<z'',z''\>=4$ and $z'''\ne 2z'$.

\vskip.04in

\noindent {\rm (iii)}  The immersion $L=\iota\circ \psi :\M\to H^m_s(-1)\subset \mathbb E^{m+1}_{s+1}$ is locally given by
  \begin{align}\label{6.2} & L(x,y)=(z(x)+w(y))\tanh\Big(\!\text{\Small$\frac{x+y}{\sqrt{2}}$}\!\Big)-\frac{z'(x)\!+\!w'(y)}{\sqrt{2}},\end{align}
where  $z$ and $w$ are curves satisfying \begin{align} &\tag{iii.1} 
\Big< (z+w)\tanh\Big(\!\text{\Small$\frac{x+y}{\sqrt{2}}$}\!\Big)-\text{\Small$\frac{z'+w'}{\sqrt{2}}$},(z+w)\tanh\Big(\!\text{\Small$\frac{x+y}{\sqrt{2}}$}\!\Big)-\text{\Small$\frac{z'+w'}{\sqrt{2}}$}\Big>=-1, 
\\&\tag{iii.2}\sqrt{2}\<z+w,2 z'-z'''\>\tanh\!\Big(\text{\Small$\frac{x+y}{\sqrt{2}}$}\Big) =\<z'+w',2z'-z'''\>,
\\&\tag{iii.3}\sqrt{2}\<z+w,2 w'-w'''\>\tanh\!\Big(\text{\Small$\frac{x+y}{\sqrt{2}}$}\Big) =\<z'+w',2w'-w'''\>.
\end{align} 

 \end{theorem}
 
\begin{proof}  Assume that $\psi: \M\to  H^m_s(-1)$ is an isometric immersion of a Lorentz surface $\M$ of constant curvature $-1$ into  $H^m_s(-1)$. 
If $M$ is totally geodesic in $H^m_s(-1)$,  we obtain (i). Hence, let us assume that $\M$ is non-totally geodesic.  

Since $\M$ is assumed to be of constant curvature $-1$, we may choose local coordinates $\{x,y\}$ such that the metric tensor is given by
\begin{align}\label{6.3}& g=-\sech^2\Big(\text{\small$\frac{x+y}{\sqrt{2}}$}\Big)(dx\otimes dy+dy\otimes dx).\end{align}
Hence, the Levi-Civita connection satisfies
\begin{equation}\begin{aligned}\label{6.4}& \nabla_{\x}\x=-\sqrt{2}\tanh \Big(\text{\small$\frac{x+y}{\sqrt{2}}$}\Big) \x,\;\; \\&\nabla_{\x}\y=0,\;\; \\&\nabla_{\y}\y=-\sqrt{2}\tanh \Big(\text{\small$\frac{x+y}{\sqrt{2}}$}\Big)\y.\end{aligned} \end{equation}

Let us put 
\begin{align}\label{6.5}& \x=\sech\Big(\text{\small$\frac{x+y}{\sqrt{2}}$}\Big) e_1 ,\;\; \y=\sech\Big(\text{\small$\frac{x+y}{\sqrt{2}}$}\Big) e_2 .\end{align}
Then we get 
\begin{align}\label{6.6}& \<e_1,e_1\>=\<e_2,e_2\>=0,\;\; \<e_1,e_2\>=-1.\end{align} 
Because $\M$ is  minimal,  it follows from \e{3.9} and \e{6.3} that  $h(e_1,e_2)=0$ holds. Hence, we may put
\begin{align}\label{6.7} h(e_1,e_1)=\xi,\; h(e_1,e_2)=0,\; h(e_2,e_2)=\eta \end{align}
for some normal vector fields $\xi,\eta$. Without loss of generality, we may assume $\xi\ne 0$.  Since $\M$ is of curvature $-1$, the equation of Gauss and \e{6.7} imply that $\<\xi,\eta\>=0$.

\vskip.04in
\noindent {\it Case} (i):  $\eta=0$. By applying formula \e{2.2} of Gauss, \e{6.3}-\e{6.5}, and \e{6.7}, we obtain
\begin{equation}\begin{aligned}\label{6.8}&L_{xx}=\sech^2\Big(\text{\small$\frac{x+y}{\sqrt{2}}$}\Big)\xi-\sqrt{2}\tanh \Big(\text{\small$\frac{x+y}{\sqrt{2}}$}\Big)L_x ,\;\;\\& L_{xy}=-\sech^2\Big(\text{\small$\frac{x+y}{\sqrt{2}}$}\Big)L,\;\;  \\&L_{yy}= -\sqrt{2}\tanh \Big(\text{\small$\frac{x+y}{\sqrt{2}}$}\Big)L_y 
.\end{aligned} \end{equation}
After solving the last two equations in \e{6.8} we have
\begin{align}\label{6.9}& L(x,y)=z(x)\tanh\Big(\text{\small$\frac{x+y}{\sqrt{2}}$}\Big)-\frac{z'(x)}{\sqrt{2}}\end{align}
for some $\mathbb E^{m+1}_{s+1}$-valued function $z$. It follows from \e{6.3}, \e{6.8}, \e{6.9} and $\<L,L\>=-1$ that $z(x)$ satisfies $\<z,z\>=0,\,\<z',z'\>=-2$, and $\<z'',z''\>=4$. Moreover,   substituting \e{6.9} into  the first equation \e{6.8} yields \begin{align} \label{6.10} \xi=\(\sqrt{2}z'(x)-\text{\small$\frac{z'''(x)}{\sqrt{2}}$}\)\cosh\Big(\text{\small$\frac{x+y}{\sqrt{2}}$}\Big).\end{align}
Combining this with $\xi\ne 0$ gives $z'''(x)\ne 2 z'(x)$.
Consequently, we obtain  (ii).

Conversely, suppose that $L$ is given by \e{6.2}, where  $z(x)$ is a timelike curve with constant speed $\sqrt{2}$ lying in the light cone $\mathcal LC$ satisfying $\<z'',z''\>=4$ and $z'''\ne 2z'$.  Then, $L$ satisfies \e{6.8} with $\xi$ given by \e{6.10}.  Moreover, from the assumption, we have
\begin{align} \label{6.11} \<z,z\>=\<z,z'\>=0,\; \<z,z''\>=-\<z',z'\>=2, \; \<z'',z''\>=4.\end{align}
Hence, we know from \e{6.9} and \e{6.11} that the induced metric tensor is given by \e{6.3}. Consequently, we see from \e{6.8} that the second fundamental form of $\psi$ satisfies $h(\x,\y)=0$. Therefore, the immersion $\psi$ is minimal.

\vskip.04in
\noindent {\it Case} (ii):  $\eta\ne 0$. By applying formula \e{2.2} of Gauss, \e{6.3}-\e{6.5} and \e{6.7}, we obtain
\begin{equation}\begin{aligned}\label{6.12}&L_{xx}=\sech^2\Big(\text{\small$\frac{x+y}{\sqrt{2}}$}\Big)\xi-\sqrt{2}\tanh \Big(\text{\small$\frac{x+y}{\sqrt{2}}$}\Big)L_x ,\;\;\\& L_{xy}=-\sech^2\Big(\text{\small$\frac{x+y}{\sqrt{2}}$}\Big)L,\;\;  \\&L_{yy}=\sech^2\Big(\text{\small$\frac{x+y}{\sqrt{2}}$}\Big)\eta-\sqrt{2}\tanh \Big(\text{\small$\frac{x+y}{\sqrt{2}}$}\Big)L_y 
.\end{aligned} \end{equation}
The compatibility conditions of \e{6.12} are given by 
\begin{align}\label{6.13}&\tilde\nabla_\y \xi=\sqrt{2}\xi \tanh\Big(\text{\small$\frac{x+y}{\sqrt{2}}$}\Big),\;\; \tilde \nabla_\x \eta=\sqrt{2}\eta \tanh\Big(\text{\small$\frac{x+y}{\sqrt{2}}$}\Big).\end{align}
Solving \e{6.13} gives $$\xi= A(x)\cosh^2\Big(\text{\small$\frac{x+y}{\sqrt{2}}$}\Big) ,\;\; \eta= B(y)\cosh^2\Big(\text{\small$\frac{x+y}{\sqrt{2}}$}\Big)$$ for some $\mathbb E^{m+1}_s$-valued functions $A(x),B(y)$ satisfying $\<A,B\>=0$. Substituting these into \e{6.12} yields
\begin{equation}\begin{aligned}\label{6.14}&L_{xx}=A(x)-\sqrt{2}\tanh \Big(\text{\small$\frac{x+y}{\sqrt{2}}$}\Big)L_x  ,\;\; \\& L_{xy}=-\sech^2 \Big(\text{\small$\frac{x+y}{\sqrt{2}}$}\Big)  L,\;\;  \\& L_{yy}=B(y)-\sqrt{2}\tanh \Big(\text{\small$\frac{x+y}{\sqrt{2}}$}\Big)L_y .\end{aligned} \end{equation}
After solving system \e{6.14} we obtain 
\begin{align}\label{6.15}&L(x,y)=(z(x)+w(y))\tanh\Big(\text{\small$\frac{x+y}{\sqrt{2}}$}\Big)-\frac{z'(x)+w'(y)}{\sqrt{2}},
\\&\label{6.16}  A(x)=\sqrt{2}z'(x)-\frac{z'''(x)}{\sqrt{2}}
,\;\;  B(y)=\sqrt{2} w'(y)-\frac{w'''(y)}{\sqrt{2}} \end{align}
for some $\mathbb E^{m+1}_s$-valued functions $z,w$.
From \e{6.15} and $\<L,L\>=-1$, we obtain condition (iii.1) of the theorem.

After differentiating  \e{6.15} we get
\begin{equation}\begin{aligned}\label{6.18}&L_x=z'(x)\tanh \Big(\text{\small$\frac{x+y}{\sqrt{2}}$}\Big)+\frac{z(x)+w(y)}{\sqrt{2}}\sech^2\Big(\text{\small$\frac{x+y}{\sqrt{2}}$}\Big)-\frac{z''(x)}{\sqrt{2}},\\&L_y=w'(y)\tanh \Big(\text{\small$\frac{x+y}{\sqrt{2}}$}\Big) + \frac{z(x)+w(y)}{\sqrt{2}}\sech^2\Big(\text{\small$\frac{x+y}{\sqrt{2}}$}\Big)-\frac{w''(y)}{\sqrt{2}}. \end{aligned} \end{equation}
Since the metric tensor of $\M$ is given by \e{6.3}, we find
\begin{align}\label{6.19}&\<L_x,L_x\>=\<L_y,L_y\>=0, \;\;\<L_x,L_y\>=-\sech^2\Big(\text{\small$\frac{x+y}{\sqrt{2}}$}\Big). \end{align}
From \e{6.14} and \e{6.16}, we obtain
\begin{equation}\begin{aligned}\label{6.20}&L_{xx}=\sqrt{2}z'(x)-\frac{z'''(x)}{\sqrt{2}}-\sqrt{2}\tanh \Big(\text{\small$\frac{x+y}{\sqrt{2}}$}\Big)L_x  ,\;\; \\& L_{xy}=-\sech^2 \Big(\text{\small$\frac{x+y}{\sqrt{2}}$}\Big)  L,\;\;  \\& L_{yy}=\sqrt{2} w'(y)-\frac{w'''(y)}{\sqrt{2}}-\sqrt{2}\tanh \Big(\text{\small$\frac{x+y}{\sqrt{2}}$}\Big)L_y .\end{aligned} \end{equation}
Because $\<L,L\>=-1$, we have $\<L_{xx},L\>=-\<L_x,L_x\>=0$. Thus, we derive from \e{6.19} and \e{6.20} that
\begin{align}\label{6.21}\sqrt{2}&\<z+w,2 z'-z'''\>\tanh\Big(\text{\small$\frac{x+y}{\sqrt{2}}$}\Big) =\<z'+w',2z'-z'''\>. \end{align}
Similarly, from $\<L_{yy},L\>=-\<L_y,L_y\>=0$, we have
\begin{align}\label{6.22}\sqrt{2}&\<z+w,2 w'-w'''\>\tanh\Big(\text{\small$\frac{x+y}{\sqrt{2}}$}\Big) =\<z'+w',2w'-w'''\>. \end{align}
These give conditions (iii.2) and (iii.3). Therefore, we obtain case (iii).

Conversely, if  $L$ is given by \e{6.2} such that $z(x),w(y)$ satisfy conditions (iii.1), (iii.2) and (iii.3), then 
we know from \e{6.2} that $L$ satisfies \e{6.20}. Also, it follows from \e{6.2} and (iii.1) that  $\<L,L\>=-1$. Thus, we have
\begin{align}\label{6.23}&\<L,L_x\>=\<L,L_y\>=0,\end{align}
which implies that
\begin{align}\label{6.24}&\<L_x,L_x\>=-\<L,L_{xx}\>,\; \<L_x,L_y\>=-\<L,L_{xy}\>,\; \<L_y,L_y\>=-\<L,L_{yy}\>.\end{align}
By applying \e{6.20}, (iii.1) and the first equation in \e{6.24}, we obtain 
\begin{align}\label{6.25}&\<L_x,L_x\>=-\<L,L_{xx}\>=\sqrt{2}\tanh \Big(\text{\small$\frac{x+y}{\sqrt{2}}$}\Big)\<L_x,L_x\>,\end{align} which yields $\<L_x,L_x\>=0$. Similarly, from \e{6.20} and (iii.3) we find $\<L_y,L_y\>=0$.
Also, after applying (iii.1), \e{6.20} and the second equation in \e{6.24}, we obtain $\<L_x,L_y\>=-\sech^2 (\text{\small$\frac{x+y}{\sqrt{2}}$}) $. Consequently, the induced metric tensor via $L$ is given by \e{6.3}. Therefore, it follows from \e{3.9} and the second equation in \e{6.20} that the immersion $\psi:\M\to H^m_s(-1)$ is minimal.
\end{proof}

\section{Explicit examples of minimal Lorentz surfaces in $S^m_s(1)$.}

There exist infinitely many spacelike curves with constant speed 2 lying in the light cone ${\mathcal LC}\subset \mathbb E^{m+1}_s$ satisfying $\<z'',z''\>=0$ and $z'''\ne 0$.

\begin{example} {\rm Consider the curve $z=z(x)$ in $\mathbb E^7_3$ defined by
\begin{equation}\begin{aligned}\notag & z(x)=\left(a \cosh px, \text{\Small$\frac{\sqrt{4r^2+a^2 p^2(p^2-r^2)}}{q\sqrt{r^2-q^2}}$} \cosh qx, \text{\Small$\frac{\sqrt{4q^2+a^2 p^2(p^2-q^2)}}{r\sqrt{r^2-q^2}}$} \sinh rx,    \right.\\& \hskip.7in a \sinh px, \text{\Small$\frac{\sqrt{4r^2+a^2 p^2(p^2-r^2)}}{q\sqrt{r^2-q^2}}$}  \sinh qx, \text{\Small$ \frac{\sqrt{4q^2+a^2 p^2(p^2-q^2)}}{r\sqrt{r^2-q^2}}$}\cosh rx,
\\&\left. \hskip1in \text{\Small$\frac{\sqrt{4(q^2+r^2)+a^2(p^2-r^2)(p^2-q^2)}}{qr}$}\right)
,\end{aligned}\end{equation}
where $a,p,q,r$ are real numbers satisfying $p>r>q> 0$. It is easy to verify that $z$ is a spacelike curve of constant speed 2 lying in $\mathcal LC$ satisfying $\<z'',z''\>=0,z'''\ne 0$.

It is direct to verify that the immersion defined by \begin{align}\notag & L(x,y)=\frac{z(x)}{x+y}-\frac{z'(x)}{2}\end{align}
  gives rise to minimal Lorentz surfaces of constant curvature one in $S^6_3(1)$. Thus, there exist infinitely many minimal Lorentz surfaces of type (b) of Theorem \ref{T:5.1}.
}\end{example}

There exist infinitely many pairs $(z,w)$ of  curves  satisfying conditions (c.1), (c.2) and (c.3) of Theorem \ref{T:5.1}. Here we provide some examples of such pairs.

\begin{example} {\rm Let $p,q,r$ be positive numbers satisfying
$$\text{\Small$ \frac{315}{4}$} p^2>80+189 r^2-64 q^2>35 p^2.$$
Consider curves $z(x)$ and $w(y)$ in $\mathbb E^{14}_6$ defined by
 \begin{equation}\begin{aligned}\notag & z(x)=\(  \text{\Small$\frac{\sqrt{256 q^2+369 r^2}}{4\sqrt{15}}$}\cosh 2x, \text{\Small$\frac{\sqrt{16q^2+609r^2}}{4\sqrt{15}}$}\sinh 4x,r\cosh 5x,0,0,0,\right.\\& \hskip.5in \left. \text{\Small$\frac{\sqrt{256 q^2+369 r^2}}{4\sqrt{15}}$}\sinh 2x,  \text{\Small$ \frac{\sqrt{16q^2+609r^2}}{4\sqrt{15}}$}\cosh 4x, r\sinh 5x,0,0,0,q,0 \),
 \\& w(y)=\(0,0,0,p \cosh\! \(\!\text{\Small$\frac{3y}{2}$}\!\),  \text{\Small$\frac{\sqrt{320+225p^2+756r^2-256q^2}}{8\sqrt{15}}$} \sinh 2y,\right. \\& \hskip.6in
  \text{\Small$ \frac{\sqrt{315p^2+1024 q^2-3024r^2-1280}}{4\sqrt{15}}$} \sinh y, 0,0,0,p \sinh\! \(\!\text{\Small$\frac{3y}{2}$}\!\), \\& \hskip.5in    \text{\Small$\frac{\sqrt{320+225p^2+756r^2-256q^2}}{8\sqrt{15}}$}\cosh 2y,
 \text{\Small$ \frac{\sqrt{315p^2+1024 q^2-3024r^2-1280}}{4\sqrt{15}}$} \cosh y, \\& \hskip.9in  \left. 0,\text{\Small$ \frac{\sqrt{320+756r^2-35p^2-256 q^2}}{8}$}\)
 .\end{aligned}\end{equation} 
 Then $z,w$ are constant speed curves lying in the light cone $\mathcal LC\subset \mathbb E^{14}_6$  satisfying 
\begin{equation}\begin{aligned}\notag & \<z',z'\>=\text{\Small$\frac{64q^2-189r^2}{20}$},\;\<w',w'\>=\text{\Small$\frac{80+189r^2-64q^2}{20}$},\; \\& \<z,w\>=\<z,z'''\>=\<z',z'''\>=\<z'',z''\>=\<w,w'''\>=\<w',w'''\>=\<w'',w''\>=0.\end{aligned}\end{equation}
Moreover, it is easy to see that conditions (c.1), (c.2) and (c.3) are satisfied. 
It is straightforward to verify that the  immersion: $$L(x,y)=\frac{z(x)+w(y)}{x+y}-\frac{z'(x)+w'(y)}{2}$$
via  $(z,w)$ defines a minimal Lorentz surface  of constant curvature one in $S^{13}_6(1)$.
}\end{example}

\section{Explicit examples of minimal Lorentz surfaces in $H^m_s(-1)$.}

There  exist infinitely many timelike curves with constant speed $\sqrt{2}$ lying in the light cone ${\mathcal LC}\subset \mathbb E^{m+1}_{s+1}$ satisfying $\<z'',z''\>=4$ and $z'''\ne 2z'$. 

\begin{example} {\rm Let $a,b,p,q$ be positive numbers satisfying 
$$p^2<\text{\small $\frac{4+a^2}{2+a^2}$}<q^2 \;\; {\rm and}\;\; b^2> \text{\small $\frac{a^2(q^2-1)(1-p^2)+2(p^2+q^2-2)}{p^2q^2}$}.  $$

Consider the curve $z=z(x)$ in $\mathbb E^8_4$ defined by
\begin{equation}\begin{aligned}\notag & z(x)=\left(b,a \cosh x,\text{\Small$\frac{\sqrt{q^2(2+a^2)-(4+a^2)}}{p\sqrt{q^2-p^2}}$} \sinh px,  \text{\Small$\frac{\sqrt{4+a^2- p^2(2+a^2)}}{q\sqrt{q^2-p^2}}$} \sinh qx,   \right.\\& \hskip.7in a \sinh px, \text{\Small$\frac{\sqrt{q^2(2+a^2)-(4+a^2)}}{p\sqrt{q^2-p^2}}$} \cosh px,  \text{\Small$\frac{\sqrt{4+a^2- p^2(2+a^2)}}{q\sqrt{q^2-p^2}}$} \cosh qx
\\&\left. \hskip1in \text{\Small$\frac{\sqrt{b^2p^2q^2-a^2(q^2-1)(1-p^2)-2(p^2+q^2-2)}}{pq}$}\right)
.\end{aligned}\end{equation}
It is easy to see that $z$ is curve lying in $\mathcal LC$  satisfying $\<z',z'\>=-2, \<z'',z''\>=4$ and $z'''\ne 2z'$.
A direct computation shows that 
\begin{align}\notag & L(x,y)=z(x)\tanh\!\Big(\text{\small$\frac{x+y}{\sqrt{2}}$}\Big)-\frac{z'(x)}{\sqrt{2}},\end{align}
 defines a minimal Lorentz surfaces of constant curvature $-1$ in $H^7_3(-1)$. Hence, there exist infinitely many minimal Lorentz surfaces of type (ii) of Theorem \ref{T:6.1}.
}\end{example}

There are many pairs $(z,w)$ of  curves  satisfying conditions  (iii.1)-(iii.3) of Theorem \ref{T:6.1}. Here we provide infinitely many examples of such pair of curves.

\begin{example} {\rm Let $a,b,p,q,r,s$ be positive numbers satisfying
 \begin{equation}\begin{aligned}\label{8.1}& a, b<1,\; \; p,q,r,s>1,\; \;   p^2<\text{\small$\frac{1}{1-b^2}$}<q^2,\; r^2<\text{\small$\frac{1}{1-a^2}$}<s^2,\\&
b^2<\text{\small$\frac{p^2+q^2-2}{p^2q^2}$},\;\;a^2<\text{\small$\frac{r^2+s^2-2}{r^2s^2}$} .\end{aligned}\end{equation}
Consider curves $z(x)$  and $w(y)$ in $\mathbb E^{14}_8$ defined by
 \begin{equation}\begin{aligned}\notag& z(x)=\Bigg(b, \text{\Small$\frac{\sqrt{p^2+q^2-b^2p^2q^2-2} }{\sqrt{(p^2-1)(q^2-1)}}$}\cosh x,\text{\Small$\frac{\sqrt{1-p^2(1-b^2)} }{\sqrt{(q^2-p^2)(q^2-1)}}$}\sinh qx, \\& \hskip.6in  \text{\Small$\frac{\sqrt{q^2(1-b^2)-1} }{\sqrt{(q^2-p^2)(p^2-1)}}$}\sinh px,
0,0,0,0,   \text{\Small$\frac{\sqrt{p^2+q^2-b^2p^2q^2-2} }{\sqrt{(p^2-1)(q^2-1)}}$}\sinh x,\\& \hskip.8in  \text{\Small$\frac{\sqrt{1-p^2(1-b^2)} }{\sqrt{(q^2-p^2)(q^2-1)}}$}\cosh qx,\text{\Small$\frac{\sqrt{q^2(1-b^2)-1} }{\sqrt{(q^2-p^2)(p^2-1)}}$}\cosh px,0,0,0\Bigg),
 \end{aligned}\end{equation}
  \begin{equation}\begin{aligned}\notag& 
 w(y)=\Bigg(0,0,0,0, a, \text{\Small$\frac{\sqrt{r^2+s^2-a^2r^2s^2-2} }{\sqrt{(r^2-1)(s^2-1)}}$}\cosh y,\text{\Small$\frac{\sqrt{1-r^2(1-a^2)} }{\sqrt{(s^2-r^2)(s^2-1)}}$}\sinh sy, \\& \hskip.6in  \text{\Small$\frac{\sqrt{s^2(1-a^2)-1} }{\sqrt{(s^2-r^2)(r^2-1)}}$}\sinh ry,
0,0,0, \text{\Small$\frac{\sqrt{r^2+s^2-a^2r^2s^2-2} }{\sqrt{(r^2-1)(s^2-1)}}$}\sinh y,\\& \hskip.8in  \text{\Small$\frac{\sqrt{1-r^2(1-a^2)} }{\sqrt{(s^2-r^2)(s^2-1)}}$}\cosh qy,\text{\Small$\frac{\sqrt{s^2(1-a^2)-1} }{\sqrt{(s^2-r^2)(r^2-1)}}$}\cosh py\Bigg).\end{aligned}\end{equation}
It is easy to verify that $z$ and $w$ satisfy conditions (iii.1), (iii.2) and (iii.3).
The associated map
 \begin{align}\notag& L(x,y)=(z(x)+w(y))\tanh\Big(\!\text{\Small$\frac{x+y}{\sqrt{2}}$}\!\Big)-\frac{z'(x)\!+\!w'(y)}{\sqrt{2}}\end{align}
 defines a minimal Lorentz surface  of constant curvature $-1$ in $H^{13}_7(-1)$.
}\end{example}

\begin{remark} There exist many  positive numbers $a,b,p,q,r,s$ satisfying the conditions given in \e{8.1}. For instance, $a=b=1/\sqrt{2}, p=r=1.1$ and $q=s=1.5$ satisfy all conditions given in \e{8.1}.\end{remark}

  \end{document}